\documentclass[12pt]{amsart}
\usepackage{amssymb}
\usepackage[mathscr]{eucal}
\usepackage{epsf}
\usepackage{epsfig}
\usepackage{color}
\DeclareGraphicsExtensions{.pstex,.eps,.epsi,.ps}

 \newtheorem{thm}{Theorem}[section]

 \newtheorem{prop}[thm]{Proposition}
 \theoremstyle{definition}
 
 \theoremstyle{remark}

\newcommand{\Real}{\mathbb{R}}

\newcommand{\dd}[1]{\frac{d}{d#1}}

\newcommand{\ddtz}{\frac{d}{dt}\Big |_{t=0}}

\newcommand{\ver}{\mathcal V}
\newcommand{\hor}{\mathcal H}

\newcommand{\hess}{\text{Hess}}

\newcommand{\RC}{\mathfrak R}

\newcommand{\marginnote}[1]
{
}

\begin{document}

\title[A Remark on the Potentials of Optimal Transport Maps]
{A Remark on the Potentials of Optimal Transport Maps}

\author{Paul W.Y. Lee}
\email{plee@math.berkeley.edu}
\address{Department of Mathematics, University of California at Berkeley, 970 Evans Hall \#3840 Berkeley, CA 94720-3840 USA}

\thanks{The author was supported by the NSERC postdoctoral fellowship.}

\begin{abstract}
Optimal maps, solutions to the optimal transportation problems, are completely determined by the corresponding $c$-convex potential functions. In this paper, we give simple sufficient conditions for a smooth function to be $c$-convex when the cost is given by minimizing a Lagrangian action. 
\end{abstract}

\maketitle


\section{Introduction}

The theory of optimal transportation starts from the problem of moving one mass to another in the most efficient way. Mathematically, the masses are given by two Borel probability measures $\mu$ and $\nu$ on a manifold $M$. The efficiency is measured by a cost function $c:M\times M\longrightarrow\Real$ and the problem is to find a Borel map which minimizes the following total cost:
\begin{equation}\label{Monge}
\int_Mc(x,\varphi(x))d\mu(x)
\end{equation}
among all Borel maps $\varphi:M\longrightarrow M$ which push $\mu$ forward to $\nu$. Here the push forward $\varphi_{*}\mu$ of a measure $\mu$ by a Borel map $\varphi$ is a measure defined by $\varphi_{*}\mu(U)=\mu(\varphi^{-1}(U))$ for all Borel sets $U\subseteq M$.

When the transportation cost $c$ is given by minimizing a Lagrangian action, the existence and uniqueness of solution to the above problem is known. More precisely, assume that $M$ is a compact connected manifold with no boundary and let $L:TM\to\Real$ be a smooth function, called Lagrangian, which satisfies the following: 
\begin{itemize}
\item the second derivative $\frac{\partial^2 L}{\partial v^2}$ is positive definite, 
\item $L$ is superlinear (ie. $\lim_{|v|\to\infty}\frac{L(x,v)}{|v|}=\infty$ and $|\cdot|$ denotes the norm of a Riemannian metric). 
\end{itemize}

Let $c$ be the cost function defined by 
\begin{equation}\label{cost}
c(x,y)=\inf\int_0^1L(\gamma(t),\dot\gamma(t))dt,
\end{equation}
where the infimum is taken over all smooth curves $\gamma(\cdot)$ connecting $x$ and $y$ (ie. $\gamma(0)=x$ and $\gamma(1)=y$). 

Assuming that the first measure $\mu$ is absolutely continuous with respect to the Lebesgue measure, Bernard-Buffoni (generalizing the earlier works of Brenier \cite{Br} in the Euclidean case and McCann \cite{Mc} in the Riemannian case) proved the existence and uniqueness of solution to the above optimal transportation problem (see also \cite{FaFi,AgLe,FiRi} for various extensions of this result and \cite{Vi1,Vi2} for a detail introduction to the theory of optimal transportation problem). 

In order to state the precise result, we need a few definitions. First let us consider the smooth function $H:T^*M\to\Real$, called Hamiltonian, defined on the cotangent bundle $T^*M$ by the Legendre transform of the Lagrangian $L$: 
\[
H(x,p)=\sup_{v\in T_xM}[p(v)-L(x,v)]. 
\] 

Let $\vec H$ be the Hamiltonian vector field on $T^*M$ defined in local coordinates $(x_1,...,x_n,p_1,...,p_n)$ by 
\[
\vec H=\sum_{i=1}^n\left(\frac{\partial H}{\partial p_i}\partial_{x_i}-\frac{\partial H}{\partial x_i}\partial_{p_i}\right)
\]
and let $\phi_t:T^*M\to T^*M$ be the flow of the Hamiltonian vector field $\vec H$. 

Another notion that we need is $c$-convexity of a function. A function $f:M\to\Real$ is $c$-convex if there is a function $\bar f:M\to\Real$ such that 
\[
f(x)=\sup_{y\in M}[\bar f(y)-c(x,y)]. 
\]
It is known that any $c$-convex function is Lipschitz (in fact locally semi-convex) and hence differentiable almost everywhere. 

\

\begin{thm}\cite{Br,Mc,BeBu}\label{BeBu}
Assume that the measure $\mu$ is absolutely continuous with respect to the Lebesgue measure and the manifold $M$ is compact. Then the optimal transportation problem corresponding to the transportation cost (\ref{cost}) has a unique solution $\varphi:M\to M$. Moreover, there exists a $c$-convex function $f:M\to\Real$ such that the solution $\varphi$ is given by 
\[
\varphi(x)=\pi(\phi_1(df_x)), 
\]
where $\pi:T^*M\to\Real$ is the canonical projection. 
\end{thm}

\

The unique solution $\varphi$ in Theorem \ref{BeBu} is called the optimal map.  One natural question after Theorem \ref{BeBu} would be the following: Is there any simple condition which guarantees a given function $f:M\to\Real$ to be  $c$-convex or whether the corresponding map $\varphi(x):=\pi(\phi_1(df_x))$ is an optimal map?

In this paper, we give a solution to the above problem using curvature type invariants introduced in \cite{AgGa} (see also Section \ref{HamCur} and Theorem \ref{main} of this paper for more detail). We state a simple consequence of the main result (Theorem \ref{main}) when the Lagrangian is natural mechanical. More precisely, let $\left<\cdot,\cdot\right>$ be a Riemannian metric on a compact manifold $M$ and let $U:M\to\Real$ be a smooth function, called potential. The next theorem is a version of Theorem \ref{main} specialized to the Lagrangians of the form $L(x,v)=\frac{1}{2}|v|^2-U(x)$, called natural mechanical Lagrangians. 

\

\begin{thm}\label{natural}
Assume that the Lagrangian $L$ is of the form $L(x,v)=\frac{1}{2}|v|^2-U(x)$, the sectional curvature is non-positive,  and the Hessian of the potential $U:M\to\Real$ satisfies  $\hess\,U\leq kI$ for some constant $k$. Let $f$ be a $C^2$ function which satisfies 
\[
\hess\,f>\begin{cases}
-\sqrt{|k|}\coth(\sqrt{|k|})I & \text{if } k<0 \\
-I & \text{if } k=0 \\
-\sqrt{|k|}\cot(\sqrt{|k|})I & \text{if } k>0 . 
\end{cases}
\]
Then $f$ is $c$-convex and the map $\varphi(x):=\pi(\phi_1(df_x))$ is the unique optimal map pushing any Borel probability measure $\mu$ forward to $\varphi_*\mu$. 
\end{thm}

\

In the Riemannian case where the potential $U\equiv 0$, Theorem \ref{main} can be improved using homogeneity of the corresponding Hamiltonian. Let $v_1=\frac{\nabla f(x)}{|\nabla f(x)|},v_2...,v_n$ be an orthonormal basis of a tangent space $T_xM$. We let $\mathcal S$ be the Hessian matrix of $f$ with respect to this basis: 
\[
\hess\,f(v_i)=\sum_{j=1}^n\mathcal S_{ij} v_j. 
\]

\

\begin{thm}\label{Riemannian}
Assume that the Lagrangian $L$ is of the form $L(x,v)=\frac{1}{2}|v|^2$ and the sectional curvature $K\leq k$ for some constant $k$. Let $f$ be a $C^2$ function and let $\lambda:=\sqrt{|k|}|\nabla f|$. Assume that $f$ satisfies 
\[
\mathcal S>
\begin{cases}
\left(\begin{array}{cc}
           -1 & 0 \\
           0 & -\lambda\coth(\lambda)I \\
         \end{array}
       \right) & k<0
\\ -I & k=0
\\\left(\begin{array}{cc}
           -1 & 0 \\
           0 & -\lambda\cot(\lambda)I \\
         \end{array}
       \right) & k>0
\end{cases}
\]
Then $f$ is $c$-convex and the map $\varphi(x):=\pi(\phi_1(df_x))$ is the unique optimal map pushing any Borel probability measure $\mu$ forward to $\varphi_*\mu$. 
\end{thm}

\

Note that if the manifold $M$ has non-positive sectional curvature, then the condition in Theorem \ref{Riemannian} is just $\hess\, f>-I$. If the manifold is two dimensional and the Riemannian metric $\left<\cdot,\cdot\right>$ is compatible with an almost complex structure $J$ (ie. there is a endomorphism $J:TM\to TM$ such that $J^2=-I$ and $\left<Ju,Jv\right>=\left<u,v\right>$), then the frame $v_1,v_2$ becomes $\frac{\nabla f}{|\nabla f|}$ and $\frac{J\nabla f}{|\nabla f|}$, respectively. Therefore, Theorem \ref{Riemannian} simplifies to the following. 

\

\begin{thm}\label{2d}
Assume that the manifold $M$ is two dimensional and there is a complex structure $J$ which is compatible with the Riemannian metric $\left<\cdot,\cdot\right>$. Let $L$ be the Lagrangian $L(x,v)=\frac{1}{2}|v|^2$ and assume that the Gauss curvature $K\leq k$ for some constant $k$. Let $f$ be a $C^2$ function and let $\lambda:=\sqrt{|k|}|\nabla f(x)|$. We define the following functions:
\[
\xi=\begin{cases}
\lambda\coth(\lambda) & \text{if } k<0 \\
1 & \text{if } k=0 \\
\lambda\cot(\lambda) & \text{if } k>0 , 
\end{cases}
\]
\[
h_1=\left<(\hess\,f(\nabla f),\nabla f\right>,
\]
\[
h_2=\left<(\hess\,f(J\nabla f),J\nabla f\right>.
\]
Assume that the following inequalities hold for all points $x$ for which $\nabla f(x)\neq 0$:
\begin{itemize}
\item  $\det\hess\,f>-|\nabla f|^2\left(\xi h_1+h_2 +\xi |\nabla f|^2\right)$,
\item  $h_1+h_2 >-(\xi+1)|\nabla f|^2$.
\end{itemize}
Then $f$ is $c$-convex and the map $\varphi(x):=\pi(\phi_1(df_x))$ is the unique optimal map pushing any Borel probability measure $\mu$ forward to $\varphi_*\mu$. 
\end{thm}

\

The paper is organized as follows. In Section \ref{HamCur}, various notions about curvature of Hamiltonian system needed in this paper is recalled. The general result (Theorem \ref{main}) mentioned above together with its consequence (Theorem \ref{natural}), are stated and proved in Section \ref{mainS}. Finally, the proof of Theorem \ref{Riemannian} is given in Section \ref{RiemannianS}. 

\bigskip

\section{Curvature of Hamiltonian System}\label{HamCur}

In this section, we recall some definitions and properties about curvature of Hamiltonian system needed in this paper (see \cite{AgGa,Ag} for a more detail discussion). 

Let $H:T^*M \to \Real$ be a Hamiltonian and let $\vec H$ be the corresponding Hamiltonian vector field defined by 
\[
\omega({\vec H},\cdot)=-dH(\cdot), 
\]
where $\omega$ is the standard symplectic structure on the cotangent bundle $T^*M$ (ie. if $(x_1,...,x_n,p_1,...,p_n)$ is a local coordinate system of $T^*M$, then $\omega$ is given by $\omega:=\sum_{i=1}^ndp_i\wedge dx_i$). 

Recall that a Lagrangian subspace of a $2n$-dimensional symplectic vector space is a $n$-dimensional subspace on which the symplectic form vanishes. Let $\alpha$ be a point of the cotangent bundle $T^*M$ and let $\ver_\alpha$ be the subspace, called the vertical space, of the tangent space $T_\alpha T^*M$ defined by 
\begin{equation}\label{verspace}
\ver_\alpha:=\{X\in T_\alpha
T^*M|d\pi(X)=0\},
\end{equation}
where $\pi:T^*M\to M$ is the cotangent bundle projection. 

The vertical space $\ver_\alpha$ is a Lagrangian subspace of the symplectic vector space $T_\alpha T^*M$. If $\phi_t$ is the flow of the Hamiltonian vector field $\vec H$, then we can define the following 1-parameter family of Lagrangian subspaces in $T_\alpha T^*M$. 
\[
J_\alpha(t)=d\phi_t^{-1}(\ver_{\phi_t(\alpha)})=\{d\phi_t^{-1}(X)|X\in\ver_{\phi_t(\alpha)}\}.
\]

The family $J_\alpha$ defines a curve, called the \emph{Jacobi curve} at the point $\alpha$, in the space of all Lagrangian subspaces contained in $T_\alpha T^*M$. For each fix time $t$, we can define an inner product $B_\alpha^t$ on each subspace $J_\alpha(t)$ by 
\[
B^{t_0}_\alpha(e,\bar e)=-\omega_\alpha(e,\dot{\bar e}(t_0)),
\]
where $\bar e(\cdot)$ is a curve in the cotangent space $T_\alpha T^*M$ such that $\bar e(t)$ is contained in $J_\alpha(t)$ for each $t$. 

The following proposition shows that if the Hamiltonian is fibrewise strictly convex, then the inner products $B^t$ are positive definite. 

\

\begin{prop}\cite{Ag}\label{canonicalbilinearform}
Assume that the restriction $H|_{T_\alpha^*M}$ of the Hamiltonian $H$ to each fibre $T^*_\alpha M$ is strictly convex. Then the bilinear form $B_\alpha^t$ is positive definite on $J_\alpha(t)$. Moreover, $B_\alpha^0$ is given by 
\[
B^0_\alpha(\partial_{p_i},\partial_{p_j})=H_{p_ip_j}.
\]
\end{prop}

\

For the rest of this paper, we assume that the Lagrangian $L$ satisfies the conditions stated at the beginning of the introduction (ie. $L$ is superlinear and the second derivative of $L$ in the fibre direction is positive definite). Under these assumptions, the Hamiltonian $H$ is fibrewise strictly convex and $B^t_\alpha$ are positive definite quadratic forms. 

Recall that a basis $\{e^1,...,e^n,f^1,...,f^n\}$ in a symplectic vector space with symplectic form $\omega$ is a Darboux basis if the following conditions are satisfied: 
\[
\omega(e^i,f^j)=\delta_{ij}, \,\omega(e^i,e^j)=0, \,\omega(f^i,f^j)=0.
\]
We can also pick a special family of basis $e^1(t),...,e^n(t)$ of the subspace $J_\alpha(t)$ which is orthonormal
with respect to the inner product $B^t_\alpha$. More precisely, we have the following proposition (see \cite{Ag} for the proof). 

\

\begin{prop}\label{Darboux}\cite{Ag}
There exists a smooth family of basis $e^1(t),...,$ $e^n(t)$ on the
vector space $J_\alpha(t)$ orthonormal with respect to the
canonical inner product $B_\alpha^t$ such that
\[
\{e^1(t),...,e^n(t),f^1(t):=-\dot e^1(t),...,f^n(t):=-\dot e^n(t)\}
\]
forms a Darboux basis of the symplectic vector space $T_\alpha
T^*M$. Moreover, if $(\bar e^1(t),...,\bar e^n(t))$ is another
such family, then there exists a constant orthogonal matrix $\mathcal O$
such that $\bar e^i(t)=\sum\limits_{j=1}^n\mathcal O_{ij}e^j(t)$.
\end{prop}

\

The 1-parameter family of Darboux basis
\[
\{e^1(t),...,e^n(t),f^1(t),...,f^n(t)\}
\]
in Proposition \ref{Darboux} is called a \emph{canonical frame} at $\alpha$. A canonical frame defines a splitting of the tangent bundle $TT^*M$. More precisely, let $\hor_\alpha$ be the subspace of the tangent space $T_\alpha T^*M$, called the horizontal space at $\alpha$, defined by 
\begin{equation}\label{horspace}
\hor_\alpha=\text{span}\{f^1(0),...,f^n(0)\}. 
\end{equation}
Then we have the splitting $T_\alpha T^*M=\ver_\alpha\oplus\hor_\alpha$ and both the vertical space $\ver_\alpha$ and the horizontal space $\hor_\alpha$ at $\alpha$ are Lagrangian subspaces. Recall that $\pi:T^*M\to M$ is the canonical projection. The restriction of the differential $d\pi$ to the horizontal space $\hor_\alpha$ gives an identification between $\hor_\alpha$ and $T_{\pi(\alpha)}M$. We let $v^\hor$ be the unique vector in the horizontal space $\hor_\alpha$ such that $d\pi(v^\hor)=v$. We call the vector $v^\hor$ the horizontal lift of $v$. 

Let $(x_1,...,x_n,p_1,...,p_n)$ be a local coordinate system of the cotangent bundle $T^*M$ and let $v$ be a tangent vector of the manifold which is given by $\sum v_i\partial_{x_i}$ in this system. Let $c_{ij}$ be structure constants defined by 
\[
v^\hor=\sum_{i=1}^nv_i\left(\partial_{x_i}+\sum_{j=1}^nc_{ij}\partial_{p_j}\right). 
\]
The following proposition gives a formula for the structure constants $c_{ij}$. 

\

\begin{prop}\label{structure}\cite{Ag}
The structure constants $c_{ij}$ satisfy
\[
\begin{split}
&2\sum_{k,l}H_{p_ip_k}c_{kl}H_{p_lp_j}\\
&=\sum_k (H_{p_k}H_{p_ip_jx_k}-H_{x_k}H_{p_ip_jp_k} -H_{p_ix_k}H_{p_kp_j}-H_{p_ip_k}H_{x_kp_j}). 
\end{split}
\]
\end{prop}

\

The next proposition gives the definition of the curvature. 

\

\begin{prop}\label{canonical}\cite{Ag}
Let $\{e^1(t),...,e^n(t),f^1(t),...,f^n(t)\}$ be a canonical frame at $\alpha$. Then there is a family of linear maps 
$R_\alpha(t):J_\alpha(t)\to J_\alpha(t)$ satisfying
\[
\dot e^i(t)=-f^i(t),\quad \dot f=R_\alpha(t)e^i(t).
\]
The operators $R_\alpha(t)$ are independent of the choice of canonical
frame and it is symmetric with respect to the positive definite quadratic form $B^t_\alpha$:  
\[
B^t_\alpha(R_\alpha(t)u,v)=B^t_\alpha(u,R_\alpha(t)v).
\] 
\end{prop}

\

The operator $R_\alpha^H:=R_\alpha(0):\ver_\alpha\to\ver_\alpha$ is called the
\emph{curvature operator} of the Hamiltonian $H$. It also satisfies the following property 
\begin{equation}\label{canonicalR}
R_\alpha(t)=d\phi_t^{-1}R^H_{\phi_t(\alpha)}d\phi_t.
\end{equation}

If $X$ is a vector in the tangent space $T_\alpha T^*M$, then the vertical part and the horizontal part of $X$ are denoted by $X_\ver$ and $X_\hor$, respectively. The following proposition gives a characterization of the curvature operator $R_\alpha^H$ without using canonical frame. 

\

\begin{prop}\label{curvaturechar}\cite{Ag}
Assume that $X$ is a vertical vector field (ie. $X(\alpha)$ is contained in $\ver_\alpha$ for each $\alpha$). Then the curvature operator $R^H$ satisfies
\[
R^H_\alpha(V(\alpha))=-[\vec H,[\vec H,V]_\hor]_\ver(\alpha).
\]
\end{prop}

\

Next, we consider the special case where the Hamiltonian is natural mechanical. More precisely, let $\left<\cdot,\cdot\right>$ be a Riemannian metric on $M$. Let $I:TM\to T^*M$ be the identification of the tangent bundle $TM$ and the cotangent bundle $T^*M$ defined by $I(v)=\left<v,\cdot\right>$. Let $U:M\to\Real$ be smooth function, called potential. A Hamiltonian is natural mechanical if it is of the following form: 
\[
H(\alpha)=\frac{1}{2}|I^{-1}(\alpha)|^2+U(x), 
\]
where $x=\pi(\alpha)$ and $\pi:T^*M\to M$ is the canonical projection. 

The structure constants $c_{ij}$ in this case is related to the Christoffel symbols as follows:

\

\begin{prop}\label{structureconstant}
Assume that the Hamiltonian is natural mechanical and let $\Gamma_{ij}^k$ be the Christoffel symbols corresponding to the given Riemannian metric and a local coordinate system $(x_1,...,x_n)$ of the manifold $M$. Then the structure constants $c_{ij}$ in the induced coordinate system $(x_1,...,x_n,p_1,...,p_n)$ of the cotangent bundle $T^*M$ are given by 
\[
c_{ij}=\sum_{k=1}^n\Gamma_{ij}^kp_k. 
\]
\end{prop}

\

\begin{proof}
Assume that the Hamiltonian is given in local coordinates $(x_1,$ $...,x_n,p_1,...,p_n)$ by 
\[
H(x,p)=\frac{1}{2}\sum_{i,j}g^{ij}p_ip_j+U(x),
\]
where $g_{ij}=\left<\partial_{x_i},\partial_{x_j}\right>$ and $g^{ij}$ denotes the inverse matrix of $g_{ij}$. 

By Proposition \ref{structure}, the structure constants $c_{ij}$ is given by 
\[
2\sum_{s,r}g^{ir}c_{rs}g^{sj}=\sum_{k,l} \left(g^{kl}\frac{\partial g^{ij}}{\partial{x_k}}p_l - g^{kj}\frac{\partial g^{li}}{\partial x_k}p_l-g^{ik}\frac{\partial g^{lj}}{\partial x_k}p_l\right). 
\]

If we rewrite the above equation, then we get the following as claimed. 
\[
c_{rs}=\frac{1}{2}\sum_{i,j} g^{ij}p_j\left(-\frac{\partial g_{rs}}{\partial{x_i}} +\frac{\partial g_{ir}}{\partial x_s}+\frac{\partial g_{is}}{\partial x_r}\right)=\sum_j \Gamma^j_{rs}p_j. 
\]

\end{proof}

For each point $\alpha$ in the cotangent bundle, we can identify the cotangent space $T^*_xM$ with the vertical space $\ver_\alpha$ by 
\begin{equation}\label{verl}
\bar\alpha\mapsto \bar\alpha^\ver:=\ddtz(\alpha+t\bar\alpha).
\end{equation}
We call $\bar\alpha^\ver$ the vertical lift of $\bar\alpha$ at the point $\alpha$. We also recall that $I:TM\to T^*M$ is the identification of the tangent and the cotangent bundle induced by the Riemannian metric. Under these identifications, the curvature $R^H$ defined above and the Riemannian curvature $\RC$  are related as follows. 

\

\begin{prop}\label{curvaturecharN}
Assume that the Hamiltonian $H$ is natural mechanical. Then the curvature $R^H$ of the Hamiltonian $H$ is given by 
\[
R^H_{Iu}((Iv)^\ver)=(I(\RC(u,v)u+\hess\, U(v)))^\ver,
\]
where $\RC$ is the Riemannian curvature and $\hess$ denotes the Hessian with respect to the given Riemannian metric. 
\end{prop}

\

\begin{proof}
Let $E:T^*M\to\Real$ be the kinetic energy Hamiltonian defined by 
\[
E(\alpha)=\frac{1}{2}|I^{-1}(\alpha)|^2. 
\]

Since $H=E+U$, we have 
\[
\vec H=\vec E-(dU)^\ver.
\]

Let $V$ be a vertical vector field. Since $[(dU)^\ver,V]$ is vertical, it follows from Proposition \ref{curvaturechar} that
\[
R^H(V)=-[\vec E-(dU)^\ver,[\vec E,V]_\hor]_\ver.
\]

By Proposition \ref{structureconstant}, the horizontal spaces of the Hamiltonians $H$ and $E$ coincide. It follows that 
\[
R^H(V)=R^E(V)+[(dU)^\ver,[\vec E,V]_\hor]_\ver.
\]

By \cite[Theorem 5.1]{AgGa}, 
\[
R^E_{Iu}((Iv)^\ver)=(I(\RC(u,v)u))^\ver. 
\]

Therefore, it remains to show that 
\[
[(dU)^\ver,[\vec E,(IX)^\ver]_\hor]_\ver=(I(\hess\, U(X)))^\ver. 
\]

A calculation using local coordinates shows that 
\[
d\pi([\vec E,(IX)^\ver])=-X.
\]

Therefore,
\[
[(dU)^\ver,[\vec E,(IX)^\ver]_\hor]_\ver=-[(dU)^\ver,X^\hor]_\ver=(I\hess\,U(X))^\ver.
\]
\end{proof}
\bigskip

\section{The Main Result}\label{mainS}

In this section, we will state and prove the main result (Theorem \ref{main}). Before that, we need a few definitions. 

Let $f:M\to\Real$ be a $C^2$ function. Its differential defines a map $P:x\mapsto df_x$ from the manifold $M$ to the cotangent bundle $T^*M$. Therefore, the differential $dP:TM\to TT^*M$ of this map sends each tangent space $T_xM$ to a $n$ dimensional Lagrangian subspace $dP(T_xM)$ of the tangent space $T_{df_x} T^*M=\ver_{df_x}\oplus\hor_{df_x}$ (see (\ref{verspace}) and (\ref{horspace}) for the definition of the vertical space $\ver_\alpha$ and the horizontal space $\hor_\alpha$, respectively). Therefore, $dP(T_xM)$ is the graph of a linear map $\hess^H f:\hor_{df_x}\to\ver_{df_x}$. We call this map the $H$-Hessian of $f$. When the Hamiltonian is natural mechanical, the $H$-Hessian coincides with the usual Hessian defined using the given Riemannian metric (see Proposition \ref{hess}). Let $e^1(t),...,e^n(t),f^1(t),...,f^n(t)$ be a canonical frame at the point $\alpha$ (see Proposition \ref{Darboux} for the definition) and let $\mathcal S$ be the matrix representation of the $H$-Hessian defined by 
\[
\hess^H f(f^i(0))=\sum_{i=1}^n\mathcal S_{ij}e^j(0).
\]
Finally we can state the main result. 

\

\begin{thm}\label{main}
Assume that the curvature $R^H$ corresponding to the Hamiltonian $H$ satisfies $R^H_\alpha\leq kI$ for some constant $k$ and all points $\alpha$ in the cotangent bundle $T^*M$. Assume that $f$ is a $C^2$ function which satisfies 
\[
\mathcal S>\begin{cases}
-\sqrt{|k|}\coth(\sqrt{|k|})I & \text{if } k<0 \\
-I & \text{if } k=0 \\
-\sqrt{|k|}\cot(\sqrt{|k|})I & \text{if } k>0. 
\end{cases}
\]
Then $f$ is $c$-convex and the $C^1$ map $\varphi(x):=\pi(\phi_1(df_x))$ is the unique optimal map pushing any Borel probability measure $\mu$ forward to $\varphi_*\mu$. 
\end{thm}

\

Before giving the proof of Theorem \ref{main}, we note that Theorem \ref{natural} is an immediate consequence of Proposition \ref{curvaturecharN}, Theorem \ref{main}, and the following two results (Proposition \ref{Rframe} and Proposition \ref{hess}). 

\

\begin{prop}\label{Rframe}
Assume that the Hamiltonian is natural mechanical. Let $v_1,...,v_n$ be an orthonormal frame of the tangent space $T_xM$. Then there is a unique canonical frame $\{e^1(t),...,e^n(t),f^1(t),...,f^n(t)\}$ satisfying 
\[
e^i(0)=(Iv_i)^\ver , \quad f^i(0)=v_i^\hor,\quad i=1,...,n.
\]
\end{prop}

\

\begin{proof}[Proof of Proposition \ref{Rframe}]
Since $v_1,...,v_n$ is orthonormal, we have, by Proposition \ref{canonicalbilinearform}, 
\[
B^0_\alpha((Iv_i)^\ver,(Iv_j)^\ver)=\delta_{ij}. 
\]

It follows from Proposition \ref{Darboux} that there exists a unique canonical frame $\{e^1(t),...,e^n(t),f^1(t),...,f^n(t)\}$ such that $e^i(0)=(Iv_i)^\ver$. 

Let $(x_1,...,x_n,p_1,...,p_n)$ be a local coordinate system of the cotangent bundle $T^*M$ around the point $\alpha$. Assume that the vectors $v$ and $w$ are given in this coordinate system by $v=\sum_{i=1}^nv_i\partial_{x_i}$ and $w=\sum_{i=1}^nw_i\partial_{x_i}$, respectively. Then  
\[
(Iv)^\ver=\sum_{i=1}^ng_{ij}v_idx_j,\quad w^\hor=\sum_{i=1}^nw_i\left(\partial_{x_i}+\sum_{j=1}^nc_{ij}\partial_{p_j}\right),
\]
where $g_{ij}=\left<v_i,v_j\right>$. 

It follows that 
\[
\omega(v^\ver,w^\hor)=\sum_{i,j=1}^ng_{ij}v_iw_j=\left<v,w\right>. 
\]

Therefore, by the definition of $e^i(0)$ and the fact that $v_1,...,v_n$ is an orthonormal family, we have 
\[
\omega(e^i(0),v_j^\hor)=\omega((Iv_i)^\ver,v_j^\hor)=\delta_{ij}.
\]

Finally since $e^1(0),...,e^n(0),f^1(0),...,f^n(0)$ is a Darboux basis,  we have $v_j^\hor=f^j(0)$ as claimed. 
\end{proof}

\

\begin{prop}\label{hess}
Assume that the Hamiltonian is natural mechanical. Then 
\[
\hess^Hf(v^\hor)=(I(\hess\, f(v)))^\ver. 
\]
\end{prop}

\

\begin{proof}[Proof of Proposition \ref{hess}]
Let $(x_1,...,x_n,p_1,...,p_n)$ be a local coordinate system and recall that $c_{ij}$ denotes the structure constants. Suppose $v$ is given, in this coordinate system, by $v=\sum v_i\partial_{x_i}$. By Proposition \ref{structureconstant},  its horizontal lift is given by
\[
v^\hor=\sum_{i} v_i\left(\partial_{x_i}+\sum_{k,j}\Gamma_{ij}^kp_k\partial_{p_j}\right),
\]
where $\Gamma_{ij}^k$ denotes the Christoffel symbols. 

In the above coordinate system, the differential $df$ is given by $\Big(x_1,...,$ $x_n,\frac{\partial f}{\partial x_1},...,\frac{\partial f}{\partial x_n}\Big)$. It follows from the definition of $H$-Hessian that 
\[
\hess^H f(v^\hor)=\sum_{i,j}\left(\frac{\partial^2 f}{\partial x_i\partial x_j}-\sum_k\Gamma_{ij}^k p_k\right)v_i\partial_{p_j}=(I(\hess\, f(v)))^\ver. 
\]
\end{proof}

\begin{proof}[Proof of Theorem \ref{main}]
Recall that $\phi_t$ denotes the flow of the Hamiltonian vector field $\vec H$. We define the maps $\Phi_t:M\to T^*M$ and $\varphi_t:M\to M$ by  
\[
\Phi_t(x)=\phi_t(df_x),\quad \varphi_t=\pi\circ\Phi_t. 
\]

The image $\Lambda_t$ of the map $d\Phi_t:T_xM\to T_{\Phi_t(x)} T^*M$ defined a Lagrangian subspace of the symplectic vector space $T_{\Phi_t(x)} T^*M$.  

Let 
\[
e^1(t),...,e^n(t),f^1(t),...,f^n(t)
\] 
be a canonical frame at the point $\Phi_0(x)$. Let $a_t$ and $b_t$ be two family of matrices defines by  
\begin{equation}\label{main1}
d\Phi_0(d\pi(f^i(0)))=\sum_{j=1}^n(a^{ij}_tf^j(t)+b^{ij}_te^j(t)).
\end{equation}

If we let $F_t=(f_1(t),...,f_n(t))^T$ and $E_t=(e_1(t),...,e_n(t))^T$, then (\ref{main1}) implies that $a_tF_t+b_tE_t$ is constant in time $t$. By Proposition \ref{canonical}, if we differentiate the above equation with respect to $t$, then we get 
\[
0=\dot a_tF_t+a_t\dot F_t+\dot b_tE_t+b_t\dot E_t=\dot a_tF_t+a_t\bar R_tE_t+\dot b_tE_t-b_tF_t, 
\]
where $\bar R_t$ denotes the matrix given by 
\begin{equation}\label{Rt}
R_t(e^i(t))=\sum_{j=1}^n(\bar R_t)_{ij}e^j(t). 
\end{equation}

It follows that 
\begin{equation}\label{R1}
\dot a_t=b_t, \quad \dot b_t=-a_t\bar R_t. 
\end{equation}

Note that the Lagrangian subspace $\Lambda_0$, and hence $\Lambda_t$ for all small enough $t$, is transversal to the vertical space $\ver_{\Phi_t(x)}$. This, in turn, is equivalent to the matrix $a_t$ being invertible for all such $t$. 
We claim that $\Lambda_t$ is transversal for all time $t$ in the interval $[0,1]$. In other words, we need to show that $a_t$ is invertible for all $t$ in $[0,1]$. If we let 
\begin{equation}\label{R2}
S_t:=a_t^{-1}b_t,
\end{equation}
then it is enough to proof that $S_t$ is bounded for all time $t$ in $[0,1]$. 

From (\ref{main1}), we have that 
\[
f^i(0)+\hess^H f(f^i(0))=\sum_{j=1}^n(a^{ij}_0f^j(0)+b^{ij}_0e^j(0)).
\]

It follows that $a_0=I$ and $\hess^H f(f^i(0))=\sum_{j=1}^nb^{ij}_0e^j(0)$. Therefore, $S_0$ satisfies the initial condition 
\[
S_0=a_0^{-1}b_0=\mathcal S. 
\] 

By (\ref{R1}) and (\ref{R2}), the matrix $S_t$ satisfies the following Riccati equation 
\[
\dot S_t+S_t^2+\bar R_t=0. 
\]

Note that the matrix $S_t$ is bounded above. For the lower bound, we need the following comparison principle of matrix Riccati equations. We denote the transpose of a matrix $B$ by $B^T$. 

\

\begin{thm}\cite{Ro}\label{compare}
Let $S^i_t$ be the solutions of the matrix Riccati equations 
\[
\dot S^i_t=A^i_t+B^i_tS^i_t+S^i_t(B^i_t)^T + S^i_tC^i_tS^i_t,\quad i=1,2.
\]
Assume that 
\[
\left(\begin{array}{cc}
           A^1_t & B^1_t \\
           (B^1_t)^T & C^1_t \\
         \end{array}
       \right)\leq \left(\begin{array}{cc}
           A^2_t & B^2_t \\
           (B^2_t)^T & C^2_t \\
         \end{array}
       \right) \quad \text{and} \quad S^1_0< S^2_0. 
\]
Then $S^1_t< S^2_t$. 
\end{thm} 

\

Therefore, by the assumption and Theorem \ref{compare}, if we consider the equation 
\begin{equation}\label{ConstRica}
\dot{\bar{S}}_t+\bar S^2_t+kI=0,  
\end{equation}
then we have 
\begin{equation}\label{comparenow}
\bar S_t\leq S_t.
\end{equation} 

The solution to (\ref{ConstRica}) is given by the following theorem.

\
 
\begin{thm}\cite{Le}\label{explicit}
Let $S_t$ be the solution of the matrix Riccati equation with constant coefficients
\[
\dot S_t=A+BS_t+S_tD + S_tCS_t. 
\]
Let 
$
M_t:=\left(\begin{array}{cc}
           M^1_t & M^2_t \\
           M^3_t & M^4_t \\
         \end{array}
       \right)
       $ be the fundamental solution of the following equation with initial condition $M_0=I$: 
\[
\dot z=\left(\begin{array}{cc}
           B & A \\
           -C & -D \\
         \end{array}
       \right)z. 
\]
Then $S_t=(M_1S_0+M_2)(M_3S_0+M_4)^{-1}$. 
\end{thm}

\

It follows that the solution to the equation (\ref{ConstRica}) with initial condition $\bar S_0=\mathcal S$ is given by 
\[
\bar S_t^k=\Gamma_1(t)(\Gamma_2(t))^{-1},
\]
where 
\[
\Gamma_1(t)=
\begin{cases}
\cosh(\sqrt{|k|} t)\mathcal S+\sqrt{|k|}\sinh(\sqrt{|k|}t)I & \text{if } k<0 \\ 
\mathcal S & \text{if } k=0 \\
\cos(\sqrt{|k|} t)\mathcal S-\sqrt{|k|}\sin(\sqrt{|k|}t)I & \text{if } k>0 
\end{cases}
\]
and
\[
\Gamma_2(t)=
\begin{cases}
\frac{\sinh(\sqrt{|k|} t)}{\sqrt{|k|}}\mathcal S+\cosh(\sqrt{|k|} t)I & \text{if } k<0 \\
t \mathcal S+I & \text{if } k=0 \\
\frac{\sin(\sqrt{|k|} t)}{\sqrt{|k|}}\mathcal S+\cos(\sqrt{|k|} t)I & \text{if } k>0 . 
\end{cases}
\]

Therefore, by (\ref{comparenow}), $S_t$ is bounded for all $t$ in $[0,1]$ if $\Gamma_2(t)>0$. This, in turn, follows from the following 
\begin{equation}\label{main2}
\mathcal S>\begin{cases}
-\sqrt{|k|}\coth(\sqrt{|k|})I & \text{if } k<0 \\
-I & \text{if } k=0 \\
-\sqrt{|k|}\cot(\sqrt{|k|})I & \text{if } k>0 . 
\end{cases}
\end{equation}

This finishes the proof of the claim that $\Gamma_t$ is transversal for all $t$ in the interval $[0,1]$. It follows from the claim and compactness of the manifold $M$ that the map $\varphi_t$ is a diffeomorphism for each $t$ in $[0,1]$.

The following theorem is proved by the method of characteristics (see \cite[Theorem 17.1, Section 17.2]{AgSa} for a proof). 

\

\begin{thm}\label{char}
Assume that $\varphi_t$ is a diffeomorphism for each time $t$ in the interval $[0,1]$. Then the curve 
\[
t\mapsto \varphi_t(x)
\]
is a strict minimizer of the minimization problem in (\ref{cost}) for each point $x$. 

Moreover, there exists $C^2$ solution $f_t$ to the Hamilton-Jacobi equation 
\[
\partial_t f_t+H(x,df_t)=0,\quad  f_0=f
\]
and it satisfies 
\[
\Phi_t(x)=(df_t)_{\varphi_t(x)}. 
\]
\end{thm}

\

Finally, we show that $\varphi_1$ is the optimal map between any measure $\mu$ and $(\varphi_1)_*\mu$. Let $\gamma(t)$ be a minimizer of (\ref{cost}) which starts from $\gamma(0)=x$ and ends at $\gamma(1)=y$. Then we have 
 
\[
f_1(y)-f_0(x)=\int_0^1\dd tf_t(\gamma(t)) dt=\int_0^1\dot f_t(\gamma(t))+df_t(\dot\gamma(t)) dt. 
\]

By the Hamilton-Jacobi equation in Theorem \ref{char}, the above equation becomes 
\begin{equation}\label{char1}
f_1(y)-f_0(x)=\int_0^1-H(\gamma(t),(df_t)_{\gamma(t)})+df_t(\dot\gamma(t)) dt. 
\end{equation}

By the definition of the Hamiltonian $H$, the above equation gives 
\begin{equation}\label{char2}
f_1(y)-f_0(x)\leq \int_0^1L(\gamma(t),\dot\gamma(t))dt= c(x,y). 
\end{equation}

By Theorem \ref{char}, we have $\dot\varphi_t(x)=\partial_pH\Big|_{df_t}(\varphi_t(x))$. Therefore, if we let $\gamma(t)=\varphi_t(x)$, then (\ref{char1}) becomes 
\[
f_1(\varphi_1(x))-f_0(x)=\int_0^1-H(x,(df_t)_{\varphi_t(x)})+df_t(\partial_pH\Big|_{df_t}(\varphi_t(x))) dt. 
\]

Finally, by the definition of the Hamiltonian $H$ and Theorem \ref{char}, the above equation gives
\[
f_1(\varphi_1(x))-f_0(x)=\int_0^1L(\varphi_t(x),\dot\varphi_t(x)) dt=c(x,\varphi_1(x)). 
\]

If we integrate both sides with respect to $\mu$, then we have 
\begin{equation}\label{char3}
\int_Mc(x,\varphi_1(x))d\mu=\int_Mf_1d(\varphi_*\mu) -\int_Mf_0d\mu. 
\end{equation}

Therefore, this finishes the proof if we combine (\ref{char2}), (\ref{char3}), and the following standard theorem in the theory of optimal transportation (see \cite{Vi2} for a proof). 

\

\begin{thm}
Let $\pi_1,\pi_2:M\times M\to M$ be the projections onto the first and the second entry, respectively. Then the following holds:
\[
\inf\int_{M\times M} c(x,y)d\Pi=\sup\int_M f_1d\nu-\int_Mf_0d\mu,
\]
where the infimum on the left is taken over all measures $\Pi$ on the product space $M\times M$ satisfying $(\pi_1)_*\Pi=\mu$ and $(\pi_2)_*\Pi=\mu$, and the supremum on the right is taken over all pairs of continuous functions $(f_0,f_1)$ satisfying $f_1(y)-f_0(x)\leq c(x,y)$. 
\end{thm}

\

\end{proof}

\bigskip

\section{The Riemannian case}\label{RiemannianS}

In the section, we specialize to the Riemannian case and give a proof of Theorem \ref{Riemannian}. In this case, we can use the homogeneity of the corresponding Hamiltonian to improve the result in Theorem \ref{main}. 

\

\begin{prop}\cite[Lemma 5.1]{AgGa}\label{homo}
Assume that the Hamiltonian $H$ is homogeneous of degree $\delta+1$ in the fibre variable. Let $\vec r$ be the Reeb field defined by $\vec r(\alpha)=\alpha^\ver$. Then $\vec r(\alpha)-t\delta\vec H(\alpha)$ is contained in $J_\alpha(t)$ for all $t$. In particular, the Hamiltonian vector field $\vec H$ is horizontal. 
\end{prop}

\

\begin{proof}[Proof of Theorem \ref{Riemannian}]
Let $\varphi$ and $\Phi$ as  in the proof of Theorem \ref{main}. Let $e^1(t),...,e^n(t),f^1(t),...,f^n(t)$ be a canonical frame at the point $\alpha=df_x$ of the cotangent bundle $T^*M$. We claim that $e^1(t)$ can be chosen to be $z(t):=\frac{1}{|I^{-1}(\alpha)|}\left(\vec r(\alpha)-t\vec H(\alpha)\right)$ if $\alpha\neq 0$. 

First, note that $z(t)$ has norm one with respect to the inner product $B_\alpha^t$. Indeed, by the definition of $B_\alpha^t$, we have 
\[
B_\alpha^t(z(t),z(t))=-\omega(z(t),\dot z(t))=\frac{\omega(\vec r(\alpha),\vec H(\alpha))}{|I^{-1}(\alpha)|^2}. 
\]

A calculation using local coordinates, we have  
\[
\omega(\vec r(\alpha),\vec H(\alpha))=|I^{-1}(\alpha)|^2
\]
and
\[
B_\alpha^t(z(t),z(t))=1. 
\]

By Theorem \ref{homo}, 
\[
z(t)=A_tE_t, 
\]
where $A_t$ is a $n\times 1$ matrix and $E_t=\left(e^1(t),...,e^n(t)\right)^T$. 

If we differentiate the above equation twice and note that $\ddot z(t)=0$, then we have 
\[
\ddot A_tE_t-2\dot A_tF_t-A_t\bar R_tE_t=0,  
\]
where $F_t=\left(f^1(t),...,f^n(t)\right)^T$ and $\bar R_t$ is defined as in (\ref{Rt}). 

It follows that $\dot A_t=0$ and $A_t$ is independent of $t$. Therefore, by Theorem \ref{canonical}, we can choose $e^1(t)$ to be $z(t)$ by applying an appropriate orthogonal transformation if $\alpha=df_x\neq 0$. 

For the rest of the proof, we assume that $e^1(t)=\frac{1}{|I^{-1}(\alpha)|}(\vec r(\alpha)-t\vec H(\alpha))$ and $f^1(t)=-\dot e^1(t)=\frac{1}{|I^{-1}(\alpha)|}\vec H(\alpha)$ whenever $\alpha\neq 0$. 

Let $\bar R^{ij}_t$ be the components of the matrix $\bar R_t$ defined as in Theorem \ref{main}. It follows that from the definition of the inner product $B_\alpha^t$ and (\ref{canonicalR}) that 
\[
\bar R^{ij}_t =B_\alpha^t(R_\alpha(t)(e^i(t)),e^j(t)) =-\omega_\alpha(d\phi_t^{-1}R^H_{\phi_t(\alpha)}d\phi_t(e^i(t)),f^j(t)).
\]

Since the symplectic form $\omega$ is preserved along the Hamiltonian flow $\phi_t$, we also have 
\[
\begin{split}
\bar R^{ij}_t
&=-\omega_{\phi_t(\alpha)}(R^H_{\phi_t(\alpha)}d\phi_t(e^i(t)),d\phi_t(f^j(t)))\\
&=B^0_{\phi_t(\alpha)}(R^H_{\phi_t(\alpha)}d\phi_t(e^i(t)),d\phi_t(e^j(t))).
\end{split}
\]

Let $v^i(t)$ be the tangent vectors defined along the geodesic $t\mapsto\pi\circ\phi_t$ by the vertical lift: $d\phi_t(e^i(t))=(Iv^i(t))^\ver$. Note that since the frame $e^1(t),...,e^n(t),f^(t),...,f^n(t)$ is a Darboux frame, $v^1(t),...,v^n(t)$ is orthonormal with respect to the Riemannian metric. It follows from Proposition \ref{canonicalbilinearform} and \ref{curvaturecharN} that 
\begin{equation}\label{Rp1}
\begin{split}
\bar R^{ij}_t
&=B^0_{\phi_t(\alpha)}(R^H_{\phi_t(\alpha)}(Iv^i(t))^\ver,(Iv^j(t))^\ver)\\
&=B^0_{\phi_t(\alpha)}((I\RC(I^{-1}(\phi_t(\alpha)),v^i(t))I^{-1}(\phi_t(\alpha)))^\ver
,(Iv^j(t))^\ver)\\
&=\left<\RC(I^{-1}(\phi_t(\alpha)),v^i(t))I^{-1}(\phi_t(\alpha)),v^j(t)\right>_{\pi(\phi_t(\alpha))}. 
\end{split}
\end{equation}

If we assume that $\nabla f(x)\neq 0$, then $e^1(t)=\frac{1}{|I^{-1}(\alpha)|}(\vec r(\alpha)-t\vec H(\alpha))$ and we have, by Proposition \ref{homo}, 
\begin{equation}\label{Rp2}
(Iv^1(t))^\ver=d\phi_t(e^1(t))=\frac{1}{|I^{-1}\alpha|}\vec r(\phi_t(\alpha))=\frac{1}{|I^{-1}\alpha|}\phi_t(\alpha)^\ver.
\end{equation}

Note that $I^{-1}\alpha=I^{-1}df=\nabla f$ and the Riemannian exponential map $\exp$ satisfies $\pi(\phi_t(Iv))=\exp(v)$. Therefore, it follows from (\ref{Rp1}) and (\ref{Rp2}) that 
\begin{equation}\label{Rp3}
\bar R^{ij}_t=|\nabla f(x)|^2\left<\RC(v^1(t),v^i(t))v^1(t),v^j(t)\right>_{\exp(t\nabla f(x))}. 
\end{equation}
Note that the equation in (\ref{Rp3}) holds also in the case $\nabla f(x)=0$. 

By the assumption of the theorem, the sectional curvature is bounded above by $k$. Therefore, the following 
\[
\left<\mathfrak R(v^1(t),\cdot)v^1(t),\cdot\right>
\] 
defines a bilinear form on the the orthogonal complement of $v^1(t)$ which is bounded above by $kI$. It follows from this and (\ref{Rp3}) that 

\begin{equation}\label{Rp4}
\bar R_t\leq \left(\begin{array}{cc}
           0 & 0 \\
           0 & k|\nabla f|^2I \\
         \end{array}
       \right). 
\end{equation}

Let $\mathcal S$ be the Hessian matrix defined by 
\[
\mathcal S_{ij}=\left<v^i(0),\hess\, f (v^j(0))\right>. 
\]
Note that $v^1(0)=\frac{\nabla f(x)}{|\nabla f(x)|}$ if $\nabla f(x)\neq 0$. 

As in the proof of Theorem \ref{main}, we want to show that the solution to the Riccati equation
\begin{equation}\label{ric1}
\dot S_t+S_t^2+\bar R_t=0, \quad S_0=\mathcal S
\end{equation}
is bounded below under the assumptions of the theorem. 

We compare the equation in (\ref{ric1}) with 
\begin{equation}\label{ric2}
\dot{\bar S}+\bar S^2+\left(\begin{array}{cc}
           0 & 0 \\
           0 & k|\nabla f|^2I \\
         \end{array}
       \right)=0, \quad S_0=\mathcal S. 
\end{equation}

By Theorem \ref{compare} and (\ref{Rp4}), we have $S_t\geq \bar S_t$. If we apply Theorem \ref{explicit}, the solution of the initial value problem in (\ref{ric2}) is given by
\[
\bar S_t=\Gamma_1(t)(\Gamma_2(t))^{-1}
\]
where $\lambda=\sqrt{|k|}|\nabla f|$, 
\[
\Gamma_1(t)=
\begin{cases}
\left(\begin{array}{cc}
           1 & 0 \\
           0 & \cosh(\lambda t)I \\
         \end{array}
       \right)\mathcal S+\left(\begin{array}{cc}
           0 & 0 \\
           0 & \lambda\sinh(\lambda t)I \\
         \end{array}
       \right) & k<0
\\ \mathcal S & k=0
\\\left(\begin{array}{cc}
           1 & 0 \\
           0 & \cos(\lambda t)I \\
         \end{array}
       \right)\mathcal S-\left(\begin{array}{cc}
           0 & 0 \\
           0 & \lambda\sin(\lambda t)I \\
         \end{array}
       \right) & k>0
       \end{cases}
\]
and 
\[
\Gamma_2(t)=
\begin{cases}
\left(\begin{array}{cc}
           t & 0 \\
           0 & \frac{\sinh(\lambda t)}{\lambda}I \\
         \end{array}
       \right)\mathcal S+\left(\begin{array}{cc}
           1 & 0 \\
           0 & \cosh(\lambda t)I \\
         \end{array}
       \right) & k<0
\\ t\mathcal S+I & k=0\\ 
\left(\begin{array}{cc}
           t & 0 \\
           0 & \frac{\sin(\lambda t)}{\lambda}I \\
         \end{array}
       \right)\mathcal S+\left(\begin{array}{cc}
           1 & 0 \\
           0 & \cos(\lambda t)I \\
         \end{array}
       \right) & k>0. 
\end{cases}
\]

Therefore, our assumption
\[
\mathcal S >
\begin{cases}
\left(\begin{array}{cc}
           -1 & 0 \\
           0 & -\lambda\coth(\lambda)I \\
         \end{array}
       \right) & k<0
\\ -I & k=0
\\\left(\begin{array}{cc}
           -1 & 0 \\
           0 & -\lambda\cot(\lambda)I \\
         \end{array}
       \right) & k>0,
\end{cases}
\]
implies that $S_t$ is bounded for all $t$ in $[0,1]$. 

The rest of the proof follows as in the proof of Theorem \ref{main}. 
\end{proof}

\end{document}